\documentclass{amsart}

\usepackage{latexsym,amsmath,amsfonts,amscd,amssymb, amsthm}

\theoremstyle{plain}
\newtheorem{theorem}{Theorem}[section]
\newtheorem{lemma}{Lemma}[section]
\newtheorem{proposition}{Proposition}[section]
\newtheorem{corollary}{Corollary}[section]

\theoremstyle{definition}
\newtheorem{definition}{Definition}[section]
\newtheorem{example}{Example}[section]

\theoremstyle{remark}
\newtheorem{remark}{Remark}[section]

\title{Rational sphere maps, linear programming, and compressed sensing}

\author{John P. D'Angelo}

\address{Dept. of Mathematics, Univ. of Illinois, 1409 W. Green St., Urbana IL 61801}

\email{jpda@math.uiuc.edu}

\author{Dusty Grundmeier}

\address{Dept. of Mathematics, Harvard }

\email{deg@math.harvard.edu}

\author{Jiri Lebl}

\address{Dept. of Mathematics, Oklahoma State Univ.}

\email{jiri.lebl@gmail.com}

\begin{document}

\maketitle

\begin{abstract} We develop a link between degree estimates for rational sphere maps and compressed sensing. We provide several new ideas and many examples, both old and new,
that amplify connections with linear programming. We close with a list of ten open problems.

\medskip

\noindent
{\bf AMS Classification Numbers}: 32H35, 32M99, 32V99, 90C05, 94A12. 

\medskip


\medskip In memory of Nick Hanges.
\end{abstract}

\section{Introduction}

This paper aims to make a surprising and powerful 
link between a basic problem in CR Geometry and the notion of compressed sensing
from applied mathematics and statistics. First briefly consider the following question. Suppose that 
${p \over q}$ is a rational mapping of degree $d$ that maps the unit sphere in the source ${\mathbb C}^n$ 
to the unit sphere in the target ${\mathbb C}^N$.
Given $d$, what is the minimum possible value of $N$? Equivalently, given $N$, what is the maximum possible value of $d$?
This problem remains open, but in the special case when $q=1$,
($p$ is a polynomial map), and the components of $p$ are orthogonal ($p$ is a monomial map), it has been completely solved.
The method of solution involves combinatorial graph theory in addition to complex variable theory,
and it does not seem to have been used outside of this problem. We modestly hope that this method will have applications
in many other circumstances, so we now describe the link with compressed sensing.

Consider an underdetermined linear system of equations $T{\bf x} = {\bf b}$ in finitely many unknowns. 
The general solution can be written ${\bf x}_0 + {\bf y}$, where $T({\bf x}_0) = {\bf b}$ and ${\bf y}$ is an arbitrary element
of the null space of $T$. In {\it compressed sensing}, one seeks a solution ${\bf a} = (a_1,..., a_n)$
with as many components $a_j$ equal to $0$ as possible. In the language of Donoho ([Do1]), one wishes 
to minimize the $L^0$ norm $\|{\bf a}\|_0$ of the solution. 
While not a norm, this name for the number of non-zero components is clever,
since $\|{\bf a}\|_0 = \sum |a_j|^0$ (if we put $0^0=0$.)  
The notion makes sense only after we have fixed a basis. In compressed sensing, several important results
compare the $L^0$ norm with the $L^1$ norm $\sum |a_j|$. See [Do2]. We discuss this matter, in our context, in Section 8.

As noted above, the study of rational sphere maps has led the authors to the same kind of question.
The simplest version involves finding the largest possible degree $d$ of a monomial mapping sending the unit sphere 
in its source ${\mathbb C}^n$ to the unit sphere in its target ${\mathbb C}^N$, in terms of $n,N$. Equivalently, given the degree $d$,
one seeks the minimum possible target dimension $N$. 
One obtains an underdetermined system of linear equations for the squared magnitudes of the coefficients 
of the distinct monomials.
This problem has been solved as follows; for $n=1$, there is no bound on $d$. For $n=2$, the sharp bound is $d \le 2N-3$.
See [DKR].
For $n\ge 3$, the sharp bound is $d \le {N-1 \over n-1}$. See [LP].
We can also regard these results as giving lower bounds on $N$ given the degree, thus making the connection
to the previous paragraph. (Minimizing $N$ is equivalent to finding the {\it sparsest} solution, given that
at least one of the coefficients of a monomial of maximum degree is not zero.)
It seems that some of the techniques developed in the papers [DKR], [G], [LP] may be both useful
and new when solving underdetermined linear systems in compressed sensing. Furthermore, in [DX2], a necessary and sufficient condition
for a map $f$ to have minimum target dimension is that a certain group $H_f$ be trivial.

Let $f= {p \over q}$ be a rational mapping in $n$ complex variables with values in ${\mathbb C}^N$ (and with no singularities on the closed unit ball).
The condition that $f$ map the unit sphere in the source
${\mathbb C}^n$ to the unit sphere in the target ${\mathbb C}^N$ arises throughout CR geometry. We call $f$ a {\it rational sphere map}.
When $f$ is not constant, $f$ is a proper holomorphic mapping 
between balls. By [F], if $n\ge 2$ and $f$ is a proper holomorphic mapping between balls with $N-n+1$ continuous derivatives
at the sphere, then $f$ is a rational sphere map. Let $\|z\|^2$ denote the squared Euclidean norm of the 
vector $z \in {\mathbb C}^k$ for any $k$
and let $\langle z,w \rangle $ denote the Euclidean inner product, again in any dimension.

In this notation we wish to find all polynomial
solutions $p$ and $q$ to the equation
$$ \|p(z)\|^2 = |q(z)|^2 \eqno (1) $$
for $z$ satisfying $\|z\|^2=1$. We may assume, without loss of generality, 
that $p$ and $q$ have no common factors and that $q(0)=1$.
We may further assume that $p(0)=0$, because the automorphism group of the target ball is transitive.
If we fix the degrees of $p$ and $q$ to be at most $d$, then we can regard
the coefficients of $p$ and $q$ as a finite number of unknowns. We obtain a {\bf quadratic} system of equations
by equating coefficients in (1) and using the condition $\|z\|^2=1$. See formula (9). When, however, 
we assume that the components of the map are monomials,
we obtain a {\bf linear} system in the squared magnitudes of the coefficients. The {\bf degree estimate} $d \le 2N-3$ from [DKR] then tells
us that the minimum $L^0$ norm $N$ of the solution is at least ${d+3 \over 2}$. Furthermore, 
this inequality is sharp, and polynomials satisfying this bound have many interesting properties. See Theorem 5.1. 
Similar conclusions apply in dimensions at least $3$, where we have $N \ge 1 + d(n-1)$.
Again this result is sharp in the monomial case. See [LP2].

We regard this paper as an invitation to study the surprising connections between two beautiful 
but not obviously related topics: degree estimates for rational 
sphere maps and the compressed sensing techniques
used in finding sparse solutions to linear systems. We therefore close the paper with a list of several open questions
that will develop and expand these connections. Several of these questions also involve the linear programming
problem of minimizing the $L^1$ norm given the degree of a monomial sphere map.

The authors acknowledge Bob Vanderbei for writing useful code concerning Problem $4$. In addition to results from [L2],
the third author has written additional independent code that corroborates the conclusions found concerning Problem 4.

\section{Easy examples}

We begin with two simple examples of the link we are emphasizing.

\begin{example} Let $p$ be a quadratic monomial map in two variables with no constant term.
Thus for complex constants $a,b,c,d,e$ we have 
$$ p(z,w) = \left (a z, bw, c z^2, dzw, e w^2\right) .$$
The condition that $p(z,w)$ maps the sphere to the sphere becomes

$$ |a|^2 + |c|^2 = |b|^2 + |e|^2 = 1$$
$$ |a|^2 + |b|^2 + |d|^2 = 2. $$
These equations are linear in the squared magnitudes, written ${\bf a} = (a_1,a_2,...,a_5)$.
We obtain the linear system
$$ \begin{pmatrix} 1 && 0 && 1 && 0 && 0 \cr 1 && 1 && 0 && 1 && 0 \cr 0 && 1 && 0 && 0 && 1  \end{pmatrix} \ \begin{pmatrix} a_1 \cr a_2 \cr a_3 \cr a_4 \cr a_5 \end{pmatrix}= 
\begin{pmatrix} 1 \cr 2 \cr 1 \end{pmatrix}. \eqno (2)$$
The linear system (2) has  general solution
$$ \begin{pmatrix} 1 \cr 1 \cr 0 \cr 0 &\cr 0\end{pmatrix} + 
a_4 \begin{pmatrix} -1 \cr 0 \cr 1 \cr 1 \cr 0 \end{pmatrix} + 
a_5 \begin{pmatrix}1 \cr 0 \cr -1 \cr 0 \cr 1 \end{pmatrix}. \eqno (3)$$
We seek solutions with nonnegative coefficients.
There is one such solution with $L^0$ norm equal to $2$, and three such solutions with $L^0$ norm equal to $3$.

We list these solutions, written in terms of the $a_j$, and then the corresponding maps.
We have chosen the coefficients of the maps to be positive; they are actually determined
only up to complex numbers of modulus $1$. 

$$ {\bf a} = (1,1,0,0,0) \ {\text {leads to}} \ (z,w, 0) $$
$$ {\bf a} = (1,0, 0,1,1) \ {\text {leads to}} \ (z,zw,w^2)$$
$$ {\bf a} = (0,1,1,1,0)\ {\text {leads to}} \ (w, zw, z^2) $$
$$ {\bf a} = (0,0,1,2,1) \ {\text {leads to}} \ (z^2, \sqrt{2}zw, w^2). $$
The first solution is first degree. There is a solution of degree $3$ to the analogous equations, but not for any degree larger than $3$.
In other words, the degree of a polynomial (or even rational)
map sending the unit sphere in ${\mathbb C}^2$ to the unit sphere
in ${\mathbb C}^3$ can be of degree at most $3$. See [Fa] for the classification of such maps. 
\end{example}

We next illustrate how the polynomial (or rational) case leads to a quadratic system.
These quadratic equations are linear in the inner products of the various unknown vectors,
and degree estimates can be regarded as a kind of quadratic compressed sensing.
We also explain how to regard the rational case as a linear problem. The following simple example from [D1]
illustrates the main idea.

\begin{example} Assume that $p:{\mathbb C}^2 \to {\mathbb C}^5$ has degree at most two and $p(0)=0$.
There are $5$ possible coefficient vectors; hence there is no loss of generality in starting with
the target dimension $5$. We want to know how {\it small} it can be. We write 
$$ p(z,w) = Az + Bw + C z^2 + Dzw + Ew^2. \eqno (4)$$
Here $A,B,C,D,E$ are elements of ${\mathbb C}^5$.
The condition that $p$ maps the source sphere to the target sphere  becomes
$$ \|Az + Bw + C z^2 + Dzw+ Ew^2\|^2 = 1 \eqno (5)$$ when  $|z|^2 + |w|^2 =1$. 
Expanding the squared norm on the left-hand side of (5) and then equating coefficients gives the following list of 
conditions on the unknowns:

$$ \langle A,B \rangle = - \langle D,E \rangle = - \langle C,D \rangle = \lambda \eqno (6.1) $$
$$ \langle A,C \rangle = \langle A,D \rangle = \langle A,E \rangle = 0 \eqno (6.2) $$
$$ \langle B,C \rangle = \langle B,D \rangle = \langle B,E \rangle = \langle C,E\rangle = 0 \eqno (6.3) $$
$$ \|A\|^2 + \|C\|^2 =1 \eqno (6.4) $$
$$ \|B\|^2 + \|E\|^2 = 1 \eqno (6.5)$$
$$ \|D\|^2 = \|C\|^2 + \|E\|^2. \eqno (6.6) $$

There are $13$ linear equations in the $15$ variables (inner products including squared norms).
The solution space to these equations involves three real parameters; we may choose them to be
$\|A\|^2$, $\|B\|^2$, and $\langle A, B \rangle = \lambda$.
We must have $ 0 \le \|A|^2 \le 1$, also that $0 \le \|B\|^2 \le 1$ and hence, by the Cauchy-Schwarz inequality,
that the complex number $\lambda$ satisfies $|\lambda| \le 1$.

The most succinct way to express this example is to write $L$ for the linear map with
$L(z,w) = Az+Bw$. Then we have
$$ p(z) = Lz \oplus \left( (\sqrt{I - L^*L}z ) \otimes z \right). $$
We summarize. The collection of polynomial mappings of degree at most two, with no constant term, and mapping the unit sphere
in ${\mathbb C}^2$ to the unit sphere in ${\mathbb C}^5$, 
is parametrized by such linear maps $L$. When $L=0$, $p$ is a homogeneous quadratic map.
When $L$ is unitary, $p$ is linear. In general $I - L^*L$ must be non-negative definite.
\end{example}

\begin{remark} A good way for solving the system resulting from (5) is to homogenize. Replace the right-hand side of (5)
with $(|z|^2 + |w|^2)^2$, expand the left-hand side, homogenize it, and finally equate coefficients.
See Section 4. \end{remark}

\section{The rational case}

Assume ${p \over q}$ is a rational mapping of degree at most $d$
that sends the unit sphere in its source to the unit sphere
in its target. Let us write 
$$ p(z) = \sum_{|\alpha| \le d} A_\alpha z^\alpha \eqno (7.1) $$
$$ q(z) = \sum_{|\beta| \le d} b_\beta z^\beta. \eqno (7.2) $$
In (7.1) the $A_\alpha$ are vectors and in (7.2) the $b_\beta$ are scalars.
The crucial condition that $\|p(z)\|^2 - |q(z)|^2 = 0$ on the unit sphere becomes
$$  \sum \left( \langle A_\alpha, A_ \gamma \rangle - b_\alpha {\overline {b_\gamma}}\right)
 z^\alpha {\overline z}^\gamma = 0 \eqno (8) $$
whenever $\sum_{j=1}^n |z_j|^2 = 1$.

It takes some work to rewrite (8) as a linear system by eliminating the constraint. 
To do so, one first replaces $z$ by $e^{i\theta}z$ in (6) and equates Fourier coefficients.
Here $\theta= (\theta_1,...,\theta_n)$ is a point on the $n$-torus. One obtains a finite number of independent equations.
Then one homogenizes each equation. Equating coefficients yields, for each multi-index $\mu$, 
$$ \sum \langle A_{\mu +\gamma}, A_{\gamma}\rangle  z^\gamma {\overline z}^\gamma \|z\|^{2d - 2|\gamma|} =
\sum b_{\mu +\gamma}{\overline b_{\gamma}}z^\gamma {\overline z}^\gamma \|z\|^{2d - 2|\gamma|}.  \eqno (9)$$
We regard this system of equations as follows. 

\medskip
{\bf Interpretation as a linear system}. 
Assume that the denominator $q$ as in (7.2) is given. We want to find all $p$ as in (7.1)
such that $\|p(z)\|^2 = |q(z)|^2$ on the unit sphere.
The right-hand side of (9) provides the right-hand side of a linear system of the form ${\bf T}(u)= v$,
where $u$ is a vector whose entries are the various inner products $\langle A_\alpha, A_\beta\rangle$.

\medskip 

In the next definition we write ${\mathbb C}^n$ for the domain of a rational map $f$. To be precise, the domain
must exclude the zero-set of the denominator $q$. The important issue is the target dimension.

\begin{definition} Let $f={p \over q}:{\mathbb C}^n \to {\mathbb C}^N$ be a rational sphere map. We say that $f$
is {\it target-minimal} if the image of $f$ lies in no affine subspace of ${\mathbb C}^N$ of lower dimension.
\end{definition}

Target minimality is equivalent to saying there is no smaller integer $k$ for which $\|f\|^2 = \|g\|^2$ for some rational map $g$
to ${\mathbb C}^k$. For example, the map $(z,w) \mapsto (z,w,0)$ is not target-minimal.
Nor is the map $(z,w) \mapsto (\alpha z, \beta z, w)$ when $|\alpha|^2 + |\beta|^2 = 1$.
Another way to express target-minimality is to expand $\|f\|^2$ in a power series about $0$; then $f$ is 
target-minimal if and only if the rank of the matrix of Taylor coefficients equals the target dimension of $f$. 
A third way to characterize
target-minimality involves groups associated with mappings. See [DX2] and Section 11.

The following non-trivial result (see [D2] and [D3]) is required in our subsequent discussion.
This result is part of a program on Hermitian analogues of Hilbert's $17$-th problem.

\begin{theorem} Suppose $q:{\mathbb C}^n \to {\mathbb C}$ is a polynomial and $q(z) \ne 0$
for $z$ in the closed unit ball. Then there is an integer $N$ and a (non-constant) polynomial map 
$p:{\mathbb C}^n \to {\mathbb C}^N$ such that both statements hold:
\begin{itemize}
\item $\|p(z)\|^2 = |q|^2 $ on the unit sphere; thus ${p \over q}$ is a rational sphere map.
\item ${p \over q}$ is reduced to lowest terms.
\end{itemize}
\end{theorem}

The second condition prevents us from considering trivial examples such as $p=q$
or $p(z) = q(z) \left(z_1,..., z_n\right)$. We also note that the result is elementary when $n=1$.

Thus given $q$, there is a smallest $N= N(q)$ for which the conclusion of the theorem applies.
This number is the smallest $N$ for which there are vectors in ${\mathbb C}^N$ whose inner products satisfy the system
of equations (9). The value of $N(q)$ depends on the coefficients of $q$; the following remark helps
indicate the subtlety.

\begin{remark}  Put $n=2$ and put $q(z_1,z_2) = 1 - \lambda z_1 z_2$.
We must have $|\lambda| < 2$ for $q$ to be non-vanishing on the closed unit ball.
Here $N(q)$ tends to infinity as $\lambda$ tends to $2$. Furthermore, the minimum possible degree of the numerator $p$ also
tends to infinity as $|\lambda|$ tends to $2$. This fact is closely related to stabilization
of Hermitian forms (see [CD]) and is explained in detail in [D3] for example. \end{remark}

Theorem 3.1 suggests how to regard the problem in a linear framework. We think of the denominator $q$ as given.
We try to find a numerator $p$ such that the conclusions of Theorem 3.1 hold.
Doing so leads to a linear system in the inner products of the unknown (vector) coefficients of $p$.
We seek the smallest dimension $N$ in which vectors exist for which this system of linear equations 
has a solution. There is no bound on the degree of $p$ nor on this dimension $N$ that depends
only on $n$ and the degree of $q$. Thus the minimum $L^0$ norm depends on the right-hand side of
the linear equation $T{\bf u} = {\bf v}$.

\section{The monomial case and homogenization }

When discussing monomial maps $p$, we generally assume, with no loss of generality, that $p(0)=0$.

The following simple idea has been used often by the authors in many papers. See for example [D1] and [LP1].
See [W] for a list (found using this idea) 
of all the monomial maps from the sphere in ${\mathbb C}^2$ to the sphere in ${\mathbb C}^5$.

Assume $n\ge 2$. Suppose that $p:{\mathbb C}^n \to {\mathbb C}^N$ is a monomial map and $\|p(z)\|^2 =1 $ when $\|z\|^2 =1$.
Put $x = (|z_1|^2,..., |z_n|^2)$. Write $s(x) = \sum_{j=1}^n x_j$. 
Write
$$ p(z) = ( ..., c_\alpha z^\alpha, ...). $$
There are several ways to create a linear system. One method is to eliminate a variable.
A second method, using homogenization, is preferable.

In the first method we write $(x_1,...,x_{n-1})= (t_1,...,t_{n-1})={\bf t}$.
The squared norms of the $c_\alpha$, written $A_\alpha$, satisfy the equation
$$ \sum_{\alpha} A_\alpha x^\alpha = 1 $$
on $s(x) = 1$. Write $\alpha = (\beta, \alpha_n)$. Replace $x_n$ by $1 - \sum_{j=1}^{n-1} x_j$ to obtain
$$  \sum_{\beta} A_\alpha {\bf t}^\beta (1 - \sum_{j=1}^{n-1} t_j)^{\alpha_n} = \sum_\mu C_\mu {\bf t}^\mu = 1. \eqno (10) $$

Equation (10) now holds for all ${\bf t}$ and we can equate coefficients.
Thus (10) yields an 
underdetermined system of linear equations for the unknown coefficients $A_\mu$. 
The coefficient of the constant term equals unity; all the other coefficients vanish.
Thus the right-hand side of the linear system is the column vector ${\bf v}$ with first coefficient $1$ and the remaining coefficients
equal to $0$. We seek a solution to the equation $T({\bf a})= {\bf v}$, where the components of ${\bf a}$ are non-negative
and where as many as possible vanish. It is also natural to fix the degree $d$ of $p$
and hence also assume that $c_\alpha \ne 0$ for some $\alpha$ of length $d$.

The system arising from (10) is awkward. Homogenization provides an alternative and more 
symmetric way to approach this system of equations.
Doing so illuminates the connection between compressed sensing and monomial sphere maps.
Let $p:{\mathbb C}^n \to {\mathbb C}^N$ be a monomial sphere map; we have
$ p(z) = (..., C_\alpha z^\alpha, ...)$
for complex numbers $C_\alpha$.
Then $p$ maps the source sphere to the target sphere if 
$$ \sum_\alpha |C_\alpha|^2 |z|^{2 \alpha}= 1 \eqno (11)$$
when $\sum |z_j|^2 =1$. Equation (11) depends only upon the $|z_j|^2$, and hence we obtain a real-variables problem by
replacing $(|z_1|^2,..., |z_n|^2)$ by $(x_1,...,x_n)$. Put $s(x) = \sum_{j=1}^n x_j$. Write
$c_\alpha = |C_\alpha|^2$. We get an equivalent equation to (11):
$$ \sum_\alpha c_\alpha x^\alpha = 1 \eqno (12)$$
on $s(x) =1$.

Following ideas from [D1], we homogenize this equation. Assume that $p$ is of degree $d$.
Homogenizing (12) yields
$$ \sum c_\alpha x^\alpha s(x)^{d - |\alpha|} = (s(x))^d = \sum_{|\beta|=d} { d \choose \beta} x^\beta. \eqno (13) $$
By homogeneity, (13) holds for all $x$. Equating coefficients yields a linear system for the unknown coefficients $c_\alpha$.
Finding the minimum target dimension $N$ is precisely the same problem as solving this linear system
with the smallest number of non-vanishing $c_\alpha$. Consider a sharp degree estimate $d \le \Phi(n,N)$. (Recall that $n \ge 2$.)
We rewrite this inequality as $N \ge \Psi(d,n)$ and reinterpret it:

\medskip

{\bf Interpretation in compressed sensing}. Of all solutions to (13), 
we seek those solutions for which the number of non-vanishing $c_\alpha$
is as small as possible, given that at least one of the coefficients of a largest degree term must be non-zero.

\begin{example} We return to quadratic monomial maps in source dimension $2$, writing $(z,w)$ for the variables
and $(x,y)$ for $(|z|^2, |w|^2)$. In this case $s(x,y)= x+y$. We assume, without loss of generality, that
$p(0)=0$. Put
$$p(z,w) = (Az,Bw, Cz^2,D zw, Ew^2). $$
Equation (13) becomes
$$  |A|^2 x(x+y) + |B|^2y(x+y) + |C|^2 x^2 + |D|^2 xy + |E|^2y^2 = x^2 + 2 xy + y^2. $$
Equating coefficients yields the system

$$  \begin{pmatrix} 1 & 0 & 1 & 0 & 0 \cr 1 & 1 & 0 & 1 & 0 \cr 0 & 1 & 0 & 0 & 1
  \end{pmatrix} \ \begin{pmatrix} a \cr b \cr c \cr d \cr e \end{pmatrix} = \begin{pmatrix} 1 \cr 2 \cr 1\end{pmatrix}.   \eqno (14) $$\end{example}

The minimum number of non-vanishing coefficients is $2$. We have the solution where $a=b=1$ 
and the other coefficients are $0$. This map, however, is degree $1$. To make it degree $2$, we require that
at least one of $c,d,e$ must be non-zero. Given this constraint, the minimum number becomes $3$.
These solutions are given by
$$ (a,b,c,d,e) = (1,0,0,1,1) $$
$$ (a,b,c,d,e) = (0,1,1,1,0)$$
$$ (a,b,c,d,e) = (0,0,1,2,1). $$

The system of equations obtained by homogenization is not the same system
as the one determined by elimination of a variable, although the corresponding maps are the same.
In particular the right-hand sides of the equation differ.
We prefer using the system obtained from homogenization.
One advantage is that the coefficients are non-negative.

The next proposition further illustrates the role of the monomial case. Recall that the {\it rank} of a 
holomorphic mapping $f:{\mathbb C}^n \to {\mathbb C}^N$
is the rank of the matrix of Taylor coefficients of $\|f(z)\|^2$. When $f$ is a monomial map, the matrix is diagonal,
and the rank is the number of non-zero diagonal elements. In general the entries of the matrix
are inner products of the unknown vector coefficients.

\begin{proposition}
Given $n, N, d$, the polynomial sphere map
$p \colon {\mathbb{C}}^n \to {\mathbb{C}^N}$ of degree $d$
that minimizes the number of
nonzero entries in the coefficient matrix of $\|p(z)\|^2$ is a
monomial map.
\end{proposition}

\begin{proof}
Consider a polynomial sphere map $p(z)$.
Replace $z$ with $e^{i\theta} z$, where $\theta$ is in the $n$-torus.
Let us average over the torus.
Let
$$ r(z,\bar{z}) = \frac{1}{{(2\pi)}^n}\int_{{[0,2\pi]}^n} \|p(e^{i\theta}z) \|^2 d \theta .$$
Note that $r(z,\bar{z}) = 1$ whenever $\|z\| = 1$.
Let us analyze the coefficient matrix of $r(z,\bar{z})$.
Averaging a coefficient $a_{\alpha\beta} z^\alpha\bar{z}^\beta$ of
$\|p(z)\|^2$ results in 0 unless $\alpha = \beta$.
If $\alpha = \beta$, the coefficient is untouched.
In other words, the coefficient matrix of $r(z,\bar{z})$ is
zero off the diagonal, and the same as the coefficient matrix of 
$\|p(z)\|^2$ on the diagonal.
In particular, $r(z,\bar{z})$ is a squared norm representing
a monomial sphere map $m(z)$.
As $p$ is of degree $d$, and the coefficient matrix of $\|p(z)\|^2$
is positive semidefinite, the diagonal, and therefore
$r(z,\bar{z}) = \| m(z) \|^2$
has a nonzero term corresponding to a term of degree $d$ in $m(z)$.
Furthermore, since they agree on the diagonals,
the coefficient matrix of $\| m(z) \|^2$ has at most as many
nonzero terms as the coefficient matrix of $\|p(z)\|^2$.
The result follows.
\end{proof}

\section{Source dimension two}

We begin with a remarkable collection of polynomials. Let $d$ be a positive integer. Define $p_d(x,y)$ by

$$   p_d(x,y) = \left({ x + \sqrt{x^2+4y} \over 2} \right)^d
+ \left({ x - \sqrt{x^2+4y} \over 2} \right)^d  + (-1)^{d+1}y^d. \eqno (15)     $$
This family of polynomials  has many interesting properties. We mention just a few now, and say more later.
See [D1], [D3], and their references.

\begin{theorem} Define $p_d$ as in (15). Each $p_d$ is a polynomial of degree $d$ and 
\begin{enumerate}
\item For each $d$, we have $p_d(x,y)=1 $ on $x+y=1$.
\item For each odd $d$, all the coefficients of $p_d$ are non-negative.
\item For each even $d$, all the coefficients of $p_d$ are non-negative except for the coefficient of $y^d$, which is $-1$.
\item $p_d(\eta x, \eta^2 y) = p_d(x,y)$ whenever $\eta$ is a $d$-th root of unity.
\item For each odd $d$, the polynomial $p_d$ has precisely ${d+3 \over 2}$ terms.
\item The polynomial $p_d$ is congruent to $x^d + y^d$ modulo $(d)$ if and only if $d=1$ or $d$ is prime.
\end{enumerate}\end{theorem}

These polynomials exhibit striking algebraic combinatorial properties. To indicate why,
instead of defining them as in (15), we could proceed as follows. Fix an odd integer $d=2r+1$ and assume that $p_d$ is of degree at most $d$. If, in addition, $p_d$ satisfies item (4)
of Theorem 5.1, then there are constants $c_k$ and $c$ such that
$$ p_d(x,y) = \sum_{k=0}^r c_k x^{2r+1-2k}y^k + c y^{2r+1}. $$ 
The reason is that only invariant monomials can arise. Next make the crucial assumption that item (1) holds. Then $c=1$ and
homogenization yields
$$ \sum_{k=0}^r c_k x^{2r+1-2k}y^k (x+y)^k+ y^{2r+1} = (x+y)^{2r+1}. $$
The corresponding linear system has a unique solution; in fact (see [D1])
$$ c_k = \left({1 \over 4}\right)^{r-k} \sum_{j=k}^r {2r+1 \choose 2j} {j \choose k}. $$
The kind of sum arising here can be evaluated using the Wilf-Zeilberger methods (see [PWZ]). Considerable work yields the formula
$$ c_k = {(2r+1)(2r-k)! \over k! (2r+1-2k)!}.  $$
Furthermore the proof of item (6) combines the method of homogenization with the analogous well-known primality property
for the polynomials $(x+y)^d$.

For the connection with sensing, the crucial point is that these polynomials have the fewest number of non-zero coefficients given their degree
and that $p(x,y)=1$ on the line $x+y=1$. We state these results next.

\begin{corollary} For each odd positive integer $2r+1$, there is a polynomial $p(x,y)$ of degree
$2r+1$, with $r+2$ terms, such that 
\begin{itemize}
\item $p(x,y)- 1$ is divisible by $x+y-1$. 
\item each coefficient of $p$ is non-negative.
\end{itemize} \end{corollary}

\begin{corollary} For each odd positive integer $d=2r+1$, 
there is a proper monomial map $f:{\mathbb B}_2 \to {\mathbb B}_{r+2}$ of degree
$2r+1$. Thus $d=2N-3$.
 \end{corollary}

By this corollary there is an example where $d = 2N-3$. The next theorem shows that $d$ can be no larger if $N$ is fixed.

\begin{theorem} [DKR] Let $p(x,y)$ be a polynomial such that $p(x,y)$ has all non-negative coefficients and $p(x,y)-1$
is divisible by $(x+y-1)$. Let $d$ denote the degree of $p$ and let $N$ denote the number of terms (distinct monomials)
in $p$. Then $d \le 2N-3$, and (by Corollary 5.1) this result is sharp. \end{theorem}

We next recast these theorems in the language of compressed sensing. The first reformulation is direct.
Using a lemma from [LL], we provide a considerably simpler second reformulation.

\medskip
{\bf First reformulation of the problem in source dimension $2$.}
\medskip

Fix a positive integer $d$. Let $K= {d^2 + 3d \over 2}$ when $d$ is odd and ${d^2 + 3d+1 \over 2}$ when $d$ is even.
We consider an unknown vector $\bf u$ in ${\mathbb R}^K$. We wish to solve a linear system
$T{\bf u}= {\bf v}$, where ${\bf v}$ is a specific vector in ${\mathbb R}^{d+1}$. We are given the standard orthonormal basis
of these Euclidean spaces. In these coordinates, the components of ${\bf v}$ are the binomial coefficients
$(1,d,{d \choose 2},..., d,1)$. We call the last $d+1$ variables in the source {\it distinguished}.
Our aim is to find ${\bf u}$ such that

\begin{itemize}
\item $T{\bf u}= {\bf v}$
\item At least one of the distinguished variables is not zero.
\item Each $u_j$ is non-negative.
\end{itemize}

Let $N$ denote the minimum $L^0$ norm of all solutions. Then the above results state that
$N= {d+3 \over 2}$ when $d$ is odd, and $N= {d \over 2} + 2$ when $d$ is even.

\begin{remark} Without the constraint that one of the distinguished variables is not $0$, then $N=2$. This solution
corresponds to $u_1 = u_2 = 1$ and $u_j = 0$ for $j \ge 2$. The corresponding polynomial is $x+y$.
For emphasis, we explain again why we have distinguished variables. We want the solution to correspond
to a polynomial of degree precisely $d$, rather than to a polynomial of degree at most $d$. \end{remark}

\begin{remark} The number $K = {d^2 + 3 d \over 2}$ is the dimension of the space of polynomials
of degree at most $d$ in two variables whose constant term is $0$. We have
$$K = 2 + 3 + ... +(d+1) = {(d+1)(d+2) \over 2} - 1 = {d^2 + 3d \over 2}. $$ 
The number $d+1$ is the dimension of the space of homogeneous polynomials of degree $d$ in two variables.
\end{remark}

\begin{example} Suppose $f$ is a monomial map from two-space of degree at most $3$.
There are $9$ possible monomials, hence $9$ unknown coefficients. Using the
method of homogenization, the system of equations that must be solved
becomes

$$ \begin{pmatrix} 1 & 0 & 1 & 0 & 0 & 1 & 0 & 0 & 0 \cr 
2 & 1 & 1 & 1 & 0 & 0 &  1 & 0 & 0 \cr
1 & 2 & 0 & 1 & 1 & 0 & 0 & 1 & 0 \cr 0 & 1 & 0 & 0 & 1 & 0 & 0 & 0 & 1  \end{pmatrix} \ 
\begin{pmatrix} a_1 \cr a_2 \cr a_3 \cr a_4 \cr a_5 \cr a_6 \cr a_7 \cr a_8 \cr a_9 \end{pmatrix} = 
\begin{pmatrix} 1 \cr 3 \cr 3 \cr 1 \end{pmatrix} \eqno (16)$$

This matrix has rank $4$, and there is a $5$-dimensional space of solutions.
Of course we want non-negative solutions. If we want something of degree $3$, then one of the last four variables must be non-zero. 
It follows that at least {\it two} of these variables must be non-zero. The smallest target dimension possible
for a map of degree $3$ is also $3$, obtained by the map
$(\sqrt{3}zw, z^3, w^3)$. In the notation from (8),
we have $a_4=3$, $a_6=1$, and $a_9 = 1$. This map is the special case of (15) when $r=1$. \end{example}

\begin{remark} If $p$ is a proper rational mapping of degree $d$
between balls, then the rank of the terms of degree $d$ is at least the domain dimension.
In particular, for a monomial map of degree $d$ in two variables, there must be at least $2$ terms of highest degree.
See Lemma 5.1 and [LL] for various generalizations. \end{remark}

\begin{example} We list the polynomials from (15) corresponding to $r=0,1,2,3,4$.
$$ x+y $$
$$ x^3 + 3xy + y^3$$
$$ x^5 + 5 x^3 y + 5 xy^2 + y^5 $$
$$ x^7 + 7 x^5 y + 14 x^3 y^2 + 7 x y^3 + y^7$$
$$ x^9 + 9 x^7 y + 27 x^5 y^2 + 30 x^3 y^3 + 9 x y^4 + y^9.  $$
A huge amount is known about these polynomials and various generalizations.
See for example Theorem 5.1, [D1], [D3], [G], and their references. \end{example}

\begin{example} Assume $p$ has at most $5$ terms and $p(x,y)=1$ on $x+y=1$. 
By Theorem 5.2, the degree of $p$ is at most $7$. 
Put $p(x,y) = a_1 x + a_2 y + a_3 x^2 + ... + a_{35} y^7$. 
Homogenizing as before yields a system of
$8$ equations in these $35$ unknowns. The sparsest solution is of course given by $x+y$; that is
$a_1 = a_2 = 1$ and all the other variables equal $0$. If we assume, however, that there is at least one term of degree $7$,
then the sparsest solution will have $5$ terms. There are four such examples:
$$ x^7 + y^7 + {7 \over 2} x^5 y + {7 \over 2} x y^5 + {7 \over 2} xy \eqno (17)$$
$$ x^7 + y^7 + 7 x^3 y^3 + 7 x y^3 + 7 x^3 y \eqno (18) $$
$$ x^7 + 7 x^5 y + 14 x^3 y^2 + 7 x y^3 + y^7 \eqno (19) $$
$$ x^7 + 7 x y^5  + 14 x^2 y^3 + 7 x^3 y + y^7. \eqno (20) $$
Notice that (19) and (20) are examples of the form found in Theorem 5.1; 
(20) is obtained from (19) by interchanging $x$ and $y$. We note that (17) and (18) are symmetric in $x$ and $y$.
Thus, in degree $7$, there are examples {\bf not} of the form found in Theorem 5.1.
The papers [DL2] and [LL] discuss this phenomenon in detail. See also Section 7. 
We repeat the main point: while the system of linear equations admits a solution
with $L^0$ norm $N$ equal to $2$, if we assume that there is a term of degree $7$, then the minimum $L^0$ norm is $5$.
\end{example}

\medskip
{\bf Second reformulation in source dimension $2$}
\medskip

In this reformulation we eliminate the notion of distinguished variables. We can do so because of a result proved about
monomial proper maps with minimum target dimension $N$. We state Lemma 3.1 from [LL], in the language of this paper.

\begin {lemma} Let $d$ be an odd integer and let $f(x,y)$ be a polynomial of degree $d$ such that
the coefficients of $f$ are non-negative and $f(x,y) = 1$ on the line $x+y=1$. Suppose also
that $f$ has $N$ terms where $N$ is minimal. (In other words, $N$ is the smallest $L^0$ norm of a solution
to the problem.) Then 
$$f(x,y) = x^d + y^d + \ {\rm {lower \ order \ terms}} .$$  \end{lemma}

In solving the linear system, we can therefore set the coefficients of $x^d$ and $y^d$ equal to $1$ and set
the rest of the distinguished variables equal to $0$. 
Doing so renders it unnecessary to discuss distinguished variables. We determine the values of $d+1$ of the variables
automatically by using this lemma.

We now get the following system. The variables $a_1,..., a_L$ are the coefficients of the monomials
$x,y, x^2, xy, y^2, x^3,x^2y,xy^2, y^3,...$ of degree at most $d-1$. We homogenize as before, except that
we set the result equal to $(x+y)^d - x^d + y^d$, obtaining two fewer equations in $d+1$ fewer variables.

We illustrate by revisiting Example 5.3. 

\begin{example} We wish to solve
$$(a_1 x + a_2 y) (x+y)^6 + (a_3 x^2 +a_4 xy + a_5 y^2 ) (x+y)^5 + ... (a_{21}x^6 + ... + a_{27} y^6)(x+y) $$
$$ = (x+y)^7 - x^7 - y^7. $$
The lemma implies that $a_{28}=a_{35}=1$ and that $a_j = 0$ for $29 \le j \le 34$.
The solutions are the same. Thus Lemma 5.1 determines the last $d+1$ variables; as claimed, the system
has $d+1$ fewer unknowns and $2$ fewer equations.
\end{example}

The paper [LL] has additional results enabling us to decrease the number of variables. For example, given that $f$
is as in Lemma 5.1, we may also assume that $a_1=a_2=0$. In fact the results of section 3 in [LL] imply
the following restrictions of the solutions to the linear system.

\begin{proposition} Let $d$ be an odd integer.  Let $f(x,y)$ be a polynomial of degree $d$ such that
the coefficients of $f$ are non-negative and $f(x,y) = 1$ on the line $x+y=1$. Assume
that $f$ has $N$ terms for $N$ as small as possible. The following all hold:
\begin{itemize}
\item The coefficient of $x^j$ is $0$ for $0 \le j \le d-1$.
\item The coefficient of $y^j$ is $0$ for $0 \le j \le d-1$.
\item The coefficient of $x^j y^k$ is $0$ for $j+k=d$ and neither $j$ nor $k$ zero.
\item The coefficient of $x^d$ is $1$ and the coefficient of $y^d$ is $1$.
\item If $d > 1$, then some coefficient of a monomial of degree $d-1$ is not zero.
\end{itemize}\end{proposition}

Using (all but the last item in) Proposition 5.1, 
we obtain a linear system with fewer unknowns. We write this new system 
\begin{itemize}
\item ${\bf T}{\bf u}= {\bf v}$.
\item Each $u_j$ is non-negative.
\end{itemize}

We also note the following result, stated explicitly in (different language in) [LL].

\begin{proposition} The solutions to this system have rational components. \end{proposition}

\section{sources and sinks}

In this section we describe the main idea from [DKR], with the modest hope that it can be used more generally
in finding the minimum $L^0$ norm of a solution to a linear system.

Let $p(x,y)$ denote a polynomial of degree $d=2r+1$ with non-negative coefficients. Assume that
$p(x,y)=1$ on the line $x+y=1$. Then the polynomial $p(x,y) - 1$ is divisible by $x+y-1$. We call the quotient $q(x,y)$.
Then $q$ is of degree $d$. Let ${\bf G}_p$ denote the Newton diagram for $q$. Thus ${\bf G}_p$ is the collection of lattice
points $(a,b)$ for which the coefficient of $x^a y^b$ in $q$ is not zero. If the coefficient is positive we label the lattice
point  with a ${\bf P}$, if it is negative we label it with a ${\bf N}$, and if zero with a ${\bf Z}$. When the point is labeled ${\bf P}$
we draw directed arrows upward and to the right from $(a,b)$ to $(a,b+1)$ and to $(a+1,b)$. When the point is labeled ${\bf N}$
we draw these arrows from $(a,b+1)$ and $(a+1,b)$ to $(a,b)$. We call a point in ${\bf G}_p$ a {\it source} 
if all arrows drawn at that point lead away from it, and we call it a {\it sink} if all arrows drawn at that point lead into it.
It is easy to see that the origin is the only source.

The first crucial observation is that the number of terms in $p$ is at least as large as the number of sinks in ${\bf G}_p$.
The reason is that a sink at $(a,b)$ means that the coefficients of $x^{a+1}y^b$ and $x^a y^{b+1}$ in $p(x,y)$ 
must be positive, as no cancellation can occur. Next one observes that the graph ${\bf G}$ corresponding to $(x+y)^d$ has each lattice point
$(a,b)$ labeled ${\bf P}$ when $a+b \le d-1$. The reason is simply that the geometric series gives
$$ { (x+y)^d - 1 \over x+y-1} = 1 + (x+y) + (x+y)^2 + ... + (x+y)^{d-1}. $$
The process of homogenization can be reversed to reach a given polynomial $p$.

What happens to the number of sinks in the directed graphs when $p$ is a solution of degree $d$ with minimum $L^0$ norm?
Assume first that $d=2r+1$ is odd. Then ${\bf G}$ has $2r+2$ sinks, the points $(a,b)$ where $a+b=d$. As we dehomogenenize,
we change the picture. By [DKR] the following hold.
The points $(d-1,0)$ and $(0,d-1)$ must have the labels ${\bf P}$ and hence there are sinks at $(d,0)$ and $(0,d)$.
These correspond to Lemma 5.1. 
The remaining $2r$ sinks may coalesce down to $r$ sinks, but no further. We conclude that the total number of sinks is
therefore at least $2 + r$, and hence the number of terms $N$ in $p$ is a least $2+r$. Thus
$N \ge 2 + r$ and hence $2r+1 = d \le 2N-3$. The group invariant polynomials from Theorem 5.1 show that this bound is realized.

We illustrate with the example $p(x,y)$ given by
$$ p(x,y) = x^5 + 5 x^3 y + 5 xy^2 + y^5. $$

We begin with $(x+y)^5$ and dehomogenize until we get to $p$. We use $\to$ to denote the effect of either rewriting terms
or setting $x+y$ equal to $1$. 

$$ (x+y)^5 = x^5 + 5 x^4 y + 10 x^3 y^2 + 10 x^2 y^3 + 5 x y^4 + y^5 \to $$
$$ x^5 + 5x^3 y(x+y) + 5x^3 y^2 + 5 x^2 y^3 + 5 x^2 y^3 + 5 x y^4 + y^5 \to $$
$$ x^5 + 5 x^3 y + 5 x^2y^2 (x+y) + 5 xy^3(x+y) + y^5 \to $$
$$ x^5 + 5 x^3 y + 5 x^2 y^2 + 5 xy^3 + y^5 \to $$
$$ x^5 + 5 x^3 y + 5 x y^2(x+y) + y^5 \to $$
$$ x^5 + 5 x^3 y + 5 x y^2 + y^5. $$

The polynomial $q$ corresponding to $p$ is given by
$$ q(x,y) = 1 + x + x^2 + x^3 + x^4 + y + y^2 + y^3 + y^4 + 2xy - 2x y^2 - xy^3 + 3x^2 y + x^2 y^2- xy^3. $$ 
There are sinks at the points $(5,0)$, $(3,1)$, $(1,2)$, and $(0,5)$.
Two of the original six sinks remain; the other four coalesced into two. The total number of sinks
in the graph ${\bf G}_p$ is four, corresponding to the minimum $L^0$ norm of the linear system.

When $d$ is even, it is easy to see that the minimum $L^0$ norm arises also from these examples. One simply multiplies either $x^{d-1}$
or $y^{d-1}$ by $x+y$ to create an example with one more term and of degree $d$.

\section{uniqueness}

For certain odd integers $d$, the only polynomials $p(x,y)$ with $N$ terms, with positive coefficients,
with $d=2N-3$, and with $p(x,y)=1$ on $x+y=1$ are those from Theorem 5.1.
One can of course interchange the roles of $x$ and $y$. For other odd integers $d$, there exist additional polynomials
with these properties. We say informally {\bf uniqueness fails}. It is natural to ask for the precise set of $d$ for which
these additional maps exist. The best results to date on this problem come from [DL2] and [LL].

All known examples where uniqueness fails come from the same procedure. It follows in all these examples, 
to be discussed below,
that the explanation for the failure of uniqueness is that we can find a solution with a smaller $L^1$ norm.

Suppose $p(x,y)$ is any polynomial with nonnegative coefficients.
Then the sum of the coefficients of $p(x,y)$ equals $p(1,1)$.
If we regard this sum as the $L^1$ norm of the solution to our linear system,
then we have a formula for the $L^1$ norm.
Consider for example the invariant polynomials $p_d$ given by (15) when $d=2r+1$ is odd. The sum of the coefficients
is
$$ \left( {1 + \sqrt{5} \over 2 }\right)^{2r+1} + \left( {1 - \sqrt{5} \over 2 }\right)^{2r+1} + 1 = \varphi^{2r+1 } + \psi^{2r+1} + 1 \eqno (21) $$
and hence essentially a Fibonacci number. 

\begin{remark} A standard formula (often erronenously attributed to Binet) for the $n$-th Fibonacci number is
$$ F_n = {1 \over \sqrt{5}} \left( \left( {1 + \sqrt{5} \over 2 }\right)^{n} - \left( {1 - \sqrt{5} \over 2 }\right)^{n}\right) = {1 \over \sqrt{5}} \left(\varphi^n - \psi^n\right). \eqno (22)$$
Since $|\psi| < 1$, for large $n=2r+1$, the $L^1$ norm (the sum of the coefficients) is close to $1+ \varphi^n$ which is close to
$\sqrt{5}F_n$. \end{remark}

All known sharp polynomials of degree $2d+1$ are found by a procedure
which replaces terms in $p_{2r+1}$ using identities arising from using analogous polynomials $p_{2k}$, each of which
has a single negative coefficient.
The resulting effect always decreases the $L^1$ norm of the solution. We give a simple example.

Begin with $p(x,y) = x^7 + 7 x^5 y +14 x^3 y^2 + 7 x y^3 + y^7$. We have $x^2+2y-y^2=1$ on the line $x+y=1$, and therefore
$x^2 + 2y = 1 + y^2$ there. On the line we have
$$ x^7 + 7 x^5 y +14 x^3 y^2 + 7 x y^3 +y^7 = x^7 + 7 x^3 y(x^2 + 2y)  + 7 xy^3 + y^7 $$
$$ \equiv x^7 + 7 x^3y(1+y^2) + 7 xy^3 + y^7 = x^7 + 7 x^3y + 7 x^3 y^3 + 7 xy^3 + y^7.   $$
Here $\equiv$ denotes equality on the line $x+y=1$.
We obtain a second polynomial of degree $7$, also with $5$ terms, also equal to $1$ on the line,
and with non-negative coefficients. Thus the $L^0$ norm is the same, but the $L^1$ norm has decreased from $30$ to $23$.
Looking back at formulas (17) through (20), we note that the corresponding $L^1$ norms are ${25/2}, 23, 30, 30$.

Let us describe this procedure in more detail. Define $p_d$ by

$$  p_d(x,y) = \left({ x + \sqrt{x^2+4y} \over 2} \right)^d
+ \left({ x - \sqrt{x^2+4y} \over 2} \right)^d +(-1)^{d+1} y^d. \eqno (25) $$
Then $p_d(x,y)= 1$ on $x+y=1$, but $p$ has a negative coefficient when $d$ is even.
Given an odd integer $d$ and an even integer $m$, the polynomial
$$ f_d(x,y) = p_d(x,y) - cx^a y^b(p_m(x,y) - 1) \eqno (26) $$
also equals $1$ on the line $x+y=1$. If we choose the monomial $cx^a y^b$ cleverly, then in passing from $p_d$ to $f_d$
we are replacing terms in $p_d$ with the same number of terms in $f_d$. In certain situations
the polynomial $f_d$ will also have nonnegative coefficients.

We can also iterate this procedure. It was proved in [DL2] that this procedure generates new sharp polynomials in various cases,
including for example when 
\begin{itemize}
\item $d \ge 7$ and $d$ is congruent to $3$ modulo $4$. 
\item $d \ge 7$ and $d$ is congruent to $1$ modulo $6$.
\end{itemize}

It is important to emphasize two points. First of all, there exist numbers for which
the only sharp polynomials are given by $p_d(x,y)$ and $p_d(y,x)$. In this case, one says {\it uniqueness holds}.
Second of all, it is unknown whether
there exist any sharp polynomials not constructed via this procedure.

\begin{remark} By the work in [LL], uniqueness is known to hold in degrees $1,3,5,9,17$. Later the third author verified it also
for $d=21$; the proof was computer assisted and took a large amount of computer time. See [L2].
It is unknown whether the collection of numbers for which uniqueness holds is finite.
\end{remark}

\section{sum of the coefficients and the $L^1$ norm}

We write ${\bf S}_d$ for the set of polynomials $f(x,y)$ of degree $d$ with non-negative coefficients such that $f(x,y)=1$ when $x+y=1$.

We saw in the last section that the sum of the coefficients, namely $f(1,1)$,  carries some useful information. 
Given that the sum of the coefficients can be regarded as the $L^1$ norm of the solution to the linear system, it is natural to gather some information
about this number. Let $m_d = \inf_{g \in {\bf S}_d} g(1,1)$. We begin with the following simple result.
The infimum is not achieved because ${\bf S}_d$ is not closed under taking limits of the coefficients.

\begin{proposition} For each degree $d$, we have $m_d =1$. For $d\ge 1$, the infimum is not attained. \end{proposition}
\begin{proof} The result is trivial when $d=0$. For $d \ge 1$ and $0 \le \lambda \le 1$, put $g(x,y) = \lambda + (1 - \lambda )(x+y)^d$.
Then $g(1,1) = \lambda + (1 -\lambda)2^d \ge 1$. Letting $\lambda$ tend to $1$ shows that the infimum is $1$. In general,
suppose $h$ is not constant, has non-negative coefficients, and satisfies $h(x,y)=1$ when $x+y=1$.
Then $h(1,1) > 1$.  \end{proof}

\begin{remark} If we restrict to polynomials with no constant term, then $m_d=2$. To see why, consider
$g(x,y) = \lambda(x+y) + (1 - \lambda )(x+y)^d$.  A similar situation arises by considering a polynomial with order of vanishing $\nu$ less than its degree $d$. On the other hand,
if $f$ is homogeneous of degree $d$, then $f(x,y) = (x+y)^d$ and $f(1,1)$ achieves the {\it maximum} possible value of $2^d$. \end{remark}

The next example illustrates a similar phenomenon.

\begin{example} For $\lambda \in [0,3]$, put 
$$f_\lambda(x,y) = x^3 + y^3 + (3 -\lambda)xy + \lambda xy (x^3 + 3 xy + y^3). $$
Then $f_\lambda \in {\bf S}_5$ for $\lambda \in (0,3]$. The degree is $5$ when  $\lambda \ne 0$. But when $\lambda=0$ the degree drops to $3$. 
Note that $f_\lambda(1,1) = 5 + 4 \lambda$. Hence, for $g$ of degree $5$ in this family of polynomials, the infimum of $g(1,1)$ is not achieved.  \end{example}

\begin{remark} Suppose $f \in {\bf S}_d$ and there are polynomials $g$ and $h$ with non-negative coefficients
such that $f(x,y)=g(x,y) + (x+y) h(x,y)$ and $h$ is not the zero polynomial.  (Thus $f$ is in the range of the partial tensor product operation. See [D3], for example, for this terminology.)
Put $u = g+h$. Then $u(x,y) =1$ when $x+y=1$ and $u(1,1) < f(1,1)$. Thus we can decrease the $L^1$ norm when such $g,h$ exist. 
\end{remark}

The following easy proposition offers a good reason for studying symmetric examples. In this result we consider a subset of ${\bf S}_d$ that is closed
under interchanging the variables and averaging. In other words $g(x,y) \in S$ implies $g(y,x) \in S$ and $f,g \in S$ implies ${f+g \over 2} \in S$.

\begin{proposition} Let $S \subset {\bf S}_d$ be closed under the operations of interchanging the variables and averaging. If $m= \inf_S g(1,1)$ is achieved,
then there is a symmetric $h \in S$ with $h(1,1)=m$. 
\end{proposition}
\begin{proof}
Suppose $m= \inf_S g(1,1)$.  Define a symmetric polynomial $h$ by
$$ h(x,y) = {g(x,y) + g(y,x) \over 2}. $$
Then $h(1,1) = g(1,1)$. Note that $h$ has non-negative coefficients, that $h(x,y) =1$ when $x+y=1$, and $h \in S$ by assumption.   \end{proof}

\begin{example} Put $g(x,y) = x^5 + y^5 + {10 \over 3} xy + {5 \over 3} (x^4 y + x y^4)$. Then
$g$ is symmetric and $g(1,1) = {26 \over 3}$. Using linear programming, one can show that $g(1,1)$ is minimal among polynomials in ${\bf S}_7$, assuming
that $g$ includes the monomials $x^7$ and $y^7$ each with coefficient $1$.
\end{example}

Proposition 8.2 is important because it significantly decreases the number of unknowns in the linear system. 
Let us make a connection with uniqueness.

\begin{remark} At the $10$th workshop on Geometric Analysis of PDEs and Several Complex Variables in August 2019, 
the first author posed a new problem. For the integers $1,3,7,19$ there are sharp examples that are also symmetric in $x,y$. 
See formulas (17) and (18) above when the degree is $7$. Excluding the trivial case
of degree $1$, the other cases provide examples where the group $\Gamma_f$ (defined in [DX1] and 
discussed in Section 11) is a semi-direct product of a torus and a group of order two. 
It is natural to ask for which integers there are sharp examples with this additional symmetry. One then considers a different
compressed sensing problem. The allowed polynomials are $(xy)^a (x^b+y^b)$ and one proceeds in a similar fashion.
In other words, one seeks constants $c[a,b]$ such that the following hold:
\begin{enumerate}
\item $ \sum_{a,b} c[a,b] (x y)^a (x^b + y^b) = 1$ on the line $x+y=1$.
\item $c[a,b] \ge 0$ for each $(a,b)$.
\item $2a + b \le d$ for each $(a,b)$ but $2a+b =d$ for some $(a,b)$.
\item The number of non-zero monomials is as small as possible. (When $b \ne 0$, a nonzero $c[a,b]$ contributes two monomials.)
\end{enumerate}
Proposition 5.1 allows us to decrease the number of variables. By Theorem 5.2, the smallest possible number of terms is ${d+3 \over 2}$.
This value is achieved when $d=1,3,7, 19$. 
The third author has given a computer assisted verification that, up to degree $31$ and except
for degrees $1,3,7,19$, the minimum number of terms {\bf exceeds} ${d+3 \over 2}$. The first author hopes to discuss this problem in future work.

\end{remark}

\section{Source dimension at least three}

When the source dimension is at least three, things are in some ways less interesting. The following results provide
monomials for which $N$ is as small as possible given $d$. 
These polynomials are both easier to find and less interesting than the polynomials
in Theorem 5.1.

\begin{proposition} Assume $n \ge 3$. For $d \ge 1$, put $N= d(n-1)+1$.
Then there is a monomial proper map $f: {\mathbb B}_n \to {\mathbb B}_N$
of degree $d = {N-1 \over n-1}$. Equivalently, there is a monomial $p$ in $n$ real variables,
with non-negative coefficients, such that $p(x)=1$ on $s(x)=1$. \end{proposition}
\begin{proof} Put $x = ({\bf  t},u)$, where $u=x_n$ and ${\bf t} = (x_1,...,x_{n-1})$.
Put $t= \sum_{j=1}^{n-1} x_j$.
Then $s(x) = t+u$. For each positive integer $d$, we define a polynomial $p$ by
$$ p(t,u) = t(1+u+u^2 + ... + u^{d-1}) + u^d.$$
Then $p$ is of degree $d$, and all of its coefficients are non-negative. 
There are precisely $d(n-1)+1$ non-vanishing coefficients. On the set where $t+u=1$ we have
$$ p(t,u) = p(t,1-t) = t\left( {1- (1-t)^d \over 1-(1-t)}\right) + (1-t)^d  = 1. $$
Replacing $(x_1,...,x_{n-1})$ by $(|z_1|^2,..., |z_{n-1}|^2)$ and $u$ by $|z_n|^2$ yields the squared norm
of the desired map $f$. 
\end{proof}

Maps as in the Proposition are called {\it Whitney maps}, because they generalize the map
$$ (x_1,...,x_n) \to (x_1,...,x_{n-1}, x_1 x_n,..., x_{n-1} x_n, x_n^2)  \eqno (W) $$
studied by Whitney. The maps in (W) are proper maps from ${\mathbb R}^n$ to ${\mathbb R}^{2n-1}$.

The following is one of the main results in [LP2]. 

\begin{theorem} Assume $n \ge 3$ and $f$ has degree $d$. Then $d \le {N-1 \over n-1}$. When $n\ge 4$, a complete list
 of the monomial sphere maps for which $d= {N-1 \over n-1}$ is known. \end{theorem}

We proceed analogously as the case $n=2$.
Assume $n \ge 3$. Fix a positive integer $d$. Let $K$ denote the dimension of the vector space of polynomials
of degree at most $d$ in $n$ variables with no constant term. Let $k$ denote the dimension of the space
of homogeneous polynomials of degree $d$ in $n$ variables. We call the last $k$ variables in ${\mathbb R}^K$ distinguished.
Let ${\bf v}$ denote the vector in ${\mathbb R}^k$
whose components are the multinomial coefficients. As before, the transformation $T$ arises via homogenization.
We consider the linear system:

\begin{itemize}
\item $T{\bf u}= {\bf v}$.
\item At least one (hence at least $n$) of the distinguished variables is not zero.
\item Each $u_j$ is non-negative.
\end{itemize}

Let $N$ denote the minimum $L^0$ norm of the solution. Combining Proposition 9.1 and Theorem 9.2 
tells us that $N = d(n-1) + 1$. The linear algebra problem again simplifies using homogenization techniques.

\section{Sparseness constraints}

The context of proper mappings between balls provides a new issue in understanding sparseness for these linear systems.
Given a proper rational mapping in source dimension $n$, not every value of $N$ arises for {\it target-minimal maps}.
We can achieve $N=1$ for non-constant rational sphere maps only when $n=1$. Thus, if $n\ge 2$, the value of $N$ cannot lie in the interval
$1 < N < n$. This fact is easy to see using complex variable theory. It is harder to see, but still true, that
the range $n < N < 2n-2$ cannot arise either. There is a general conjecture
on the gaps that are possible. See [HJX] and [HJY] for work on the {\it gap conjecture}.

We state and sketch the proof of a result from [DL1]. Fix the source dimension $n$. If $N \ge n^2 - 2n+2$, there is a target-minimal monomial example with $L^0$ norm equal to $N$.

\begin{theorem} Put $T(n)= n^2-2n+2$. For each $N$ with $N \ge T(n)$, there is a target-minimal proper monomial
map $f:{\mathbb B}_n \to {\mathbb B}_N$. \end{theorem}
\begin{proof} The proof for $n=1$ is easy. We simply take an $N$-tuple of positive numbers $c_j^2$ whose sum is $1$.
Use these numbers to define $f$ by
$$f(z) = (c_1 z,c_2 z_2,...,c_N z^N). $$
For $n\ge 2$ we proceed as follows. For $x \in {\mathbb R}^m$, put $s(x)= \sum x_j$.
Given a polynomial $p(x)$ with non-negative coefficients and with $p(x)=1$ on the line $s(x)=1$, and $c>0$, we may define new polynomials
with the same properties by 
$$ Wp(x) = p(x) - c x_n^d + c x_n^d s(x) $$
$$ Vp(x) = p(x) - {c \over 2} x_n^d + {c \over 2} x_n^d s(x).$$
Regard $W$ and $V$ as operations which we may iterate, always applying them on the pure term of highest degree, we form
$V^kW^js(x)$. One can then show for $N\ge T(n)$ that we obtain an example with the desired number of terms. See [DL1] for details.
\end{proof}

We mention several related results.  Let $p$ be a rational sphere map of degree $d$ and
source dimension $n$. Then the terms of degree exactly equal to $d$ must map into a subspace of dimension at least $n$.
On the other hand, if $p$ is a {\it homogeneous} polynomial sphere map of degree $d$, then the target dimension of $p$ must be at least
the binomial coefficient ${n+d-1 \choose d}$. The maps from Corollary 5.2 and Proposition 9.1 illustrate an interesting phenomenon.
For these maps, which are not homogeneous for $d\ge 3$,
the terms of highest degree map into a space of dimension at least (the source dimension) $n$, but no larger. For a typical
rational sphere map, the terms of highest degree themselves already map into a space too large 
for the target dimension to be as small as possible.
In other words, the solution of the compressed sensing problem must have enough terms of degree $d$ but not too many.

\section{Groups associated with mappings}

The papers [DX1]  and [DX2] and the book [D3] associate groups with holomorphic mappings. These groups have many uses; we mention only those
that directly bear on our discussion. Let ${\rm Aut}({\mathbb B}_n)$ denote the group of holomorphic automorphsims of the ball
${\mathbb B}_n$.

\begin{definition} Let $f:{\mathbb B}_n \to {\mathbb B}_N$ be a proper rational mapping. Then $A_f$ is the subgroup
of ${\rm Aut}({\mathbb B}_n) \times {\rm Aut} ({\mathbb B_N})$ consisting of those pairs $(\gamma, \psi)$ for which
$$ f \circ \gamma = \psi \circ f. \eqno (29) $$
The Hermitian source group $\Gamma_f$ is the projection of $A_f$ onto its first factor; the Hermitian target group $T_f$
is the projection onto the second factor. 
\end{definition}

It follows from the work in [DX2] that $f$ is target-minimal if and only if there is a group homomorphism for $\Gamma_f$ to $T_f$.
When $f$ is target-minimal, and $\gamma \in \Gamma_f$, then the automorphism $\psi = \Phi(\gamma)$ from (29) is uniquely determined.
The map $\Phi$ is easily seen to be a group homomorphism. When $f$ is not target-minimal, the automorphism $\psi$ is not uniquely determined.
Therefore the groups provide information about the minimum target dimension $N$, the primary issue in this paper.

\begin{example} 
The maps defined in (15) are all target-minimal. We compute this homomorphism in the degree $3$ case. The other cases are similar.
Put $f(z,w) = (z^3, \sqrt{3}zw, w^3)$. Then $f:{\mathbb B}_2 \to {\mathbb B}_3$ is a proper
holomorphic map. Here $\Gamma_f$ is the group generated by the diagonal unitary matrices and the element of order two
interchanging the variables. Here $T_f$ is the group generated by 
$$ \begin{pmatrix} e^{i 3 \alpha} & 0 & 0 \cr 0 & e^{i(\alpha+\beta)} & 0 \cr 0 & 0 & e^{i 3 \beta} \end{pmatrix}$$

$$ \begin{pmatrix} 0 & 0 & 1 \cr 0 & 1 & 0 \cr 1 & 0 & 0 \end{pmatrix}. $$
The kernel of $\Phi$ is the cyclic group of order three generated by
$$ \begin{pmatrix} \eta & 0 \cr 0 & \eta^2 \end{pmatrix}$$
where $\eta$ is a primitive cube root of $1$. Here ${\rm kernel}(\Phi)$ is the set of $\gamma$ for which $f \circ \gamma = f$.
\end{example}

These groups can be used to decide when a proper map between balls is equivalent to a monomial map. We have the following general results:

\begin{theorem} Let $f:{\mathbb B}_n \to {\mathbb B}_N$ be a  rational proper mapping. Then:

\begin{itemize} 

\item $\Gamma_f = {\rm Aut}({\mathbb B}_n)$ if and only if $f$ is a linear fractional transformation.

\item $\Gamma_f$ is noncompact if and only if $\Gamma_f = {\rm Aut}({\mathbb B}_n)$. Otherwise $\Gamma_f$
is contained in a conjugate of ${\bf U}(n)$.

\item $\Gamma_f$ is a conjugate of ${\bf U}(n)$ if and only if $f$ is equivalent to a juxtaposition of tensor powers.

\item $\Gamma_f = {\bf U}(n)$ if and only if $f$ is a juxtaposition of tensor powers. 

\item $\Gamma_f$ contains an $n$-torus if and only if $f$ is equivalent to a monomial map.

\item Let $G$ be a finite subgroup of ${\rm Aut}({\mathbb B}_n)$. Then there is an $N$ and a proper rational map $f:{\mathbb B}_n \to {\mathbb B}_N$ such that
$\Gamma_f = G$.

\item Let $G$ be a finite subgroup of ${\bf U}(n)$. Then there is an $N$ and a proper polynomial map $f:{\mathbb B}_n \to {\mathbb B}_N$ such that
$\Gamma_f = G$.

\end{itemize}

\end{theorem}

The last two parts of this theorem suggest additional questions of the sort considered throughout this paper. 
Given a finite subgroup $G$ of the unitary group, there is a map $f$ for which $\Gamma_f=G$. 
Relating either the smallest possible degree or the minimum possible target dimension to the group seems difficult. 
In [D3] it is noted that a map $f$ for which $\Gamma_f$ is trivial must be of degree at least $3$ and that $3$ is possible.

\section{Open problems}

\begin{enumerate} 
\item Assume $n\ge 2$. Let $f:{\mathbb C}^n \to {\mathbb C}^N$ be a rational sphere map of degree $d$. When $n=2$, prove that
$d \le 2N-3$. When $n\ge 3$, prove that $d \le {N-1 \over n-1}$.

\item Put $n=2$. For each odd degree $d$, find the minimum $N=N(d)$ for which there is 
a monomial sphere map $f:{\mathbb C}^2 \to {\mathbb C}^N$ of degree $d$ with
$\|f(z,w)\|^2 = \|f(w,z)\|^2$. For example, when $d=1,3,7,19$ the answer is ${d+3 \over 2}$ but for other small odd $d$ this 
minimum cannot be achieved. Also, is the number of odd $d$ for which $N(d)= {d+3 \over 2}$ finite or infinite?
See Remark 8.3 for more discussion.

\item Put $n=2$. Suppose that $p(x,y)$ is a sharp polynomial of degree $d$. What is the minimum value of the $L^1$ norm $p(1,1)$? 

\item Let $S$ be the largest subset of ${\bf S}_d$ such that each $f\in S$ contains the monomials $x^d$ and $y^d$ with coefficient $1$.
Find the minimum value of $f(1,1)$. See Example 8.2 and Remark 8.3. In other words, 
find the minimum $L^1$ norm of all solutions to the sensing problem. Stated otherwise, given the degree $d$, 
we wish to minimize a linear function of the coefficients $c[a,b]$ subject
to the first three constraints from Remark 8.3. Proposition 8.2 considerably reduces the search space. 
In addition, can one characterize the polynomials realizing the minimum? The first author has obtained some results about the structure of
this problem and the third author has written code finding all examples up to degree $35$. As noted in the introduction, Bob Vanderbei has written independent code.
The formulas are absurdly complicated. For example, in degree $11$
the unique minimizer is given by $ x^{11} + y^{11} + (xy) g(x,y)$
where
$$ g(x,y) = {99 \over 28} + {33 \over 14} (x^4+y^4) + {33 \over 14}(x^5 + y^5)+ {55 \over 28}(x^8 + y^8) + {11 \over 14}(x^9 + y^9). $$

\item If the previous problems are too difficult, can one find asymptotic relations between the $L^1$ and $L^0$ norms as the degree increases? Compare with [Do2]. 

\item Find all sharp monomial sphere maps when $n=3$; the answer is known for $n\ge 4$. See Proposition 9.1, Theorem 9.1,  and [LP]. 

\item Express the results from  [BEH] and [BH] in the language of compressed sensing.

\item Extend the ideas of this paper to the hyperquadric setting. See for example [Gr], [GLV], [LP1], and [LP2].

\item Relate the gap conjecture of Huang-Ji (See [HJX] and [HJY]) to compressed sensing.

\item Given a finite subgroup $G$ of ${\bf U}(n)$, what is the minimum possible degree of a rational sphere map $f$ with $\Gamma_f = G$?
Can one relate this problem to known results in invariant theory?

\end{enumerate}

\section{bibliography}

\medskip

[BEH] M. S. Baouendi, P. F. Ebenfelt, and X. Huang, 
Holomorphic mappings between hyperquadrics with small signature difference, {\it Amer. J. Math.} 133 (2011), no. 6, 1633-1661. 

\medskip

[BH] M. S. Baouendi and X. Huang, Super-rigidity for holomorphic mappings between hyperquadrics with positive signature.,
{\it J. Differential Geom.} 69 (2005), no. 2, 379-98. 

\medskip

[CD]  D. W. Catlin and J. P. D'Angelo,  A stabilization theorem for Hermitian forms and applications to holomorphic mappings, 
{\it Math. Res. Lett.} 3 (1996), no. 2, 149-166.

\medskip

[D1] J. D'Angelo,  Several Complex Variables and the Geometry of Real Hypersurfaces,
CRC Press, Boca Raton, Fla., 1992.

\medskip

[D2] J. D'Angelo, Proper holomorphic mappings,
positivity conditions, and isometric imbedding, {\it J. Korean Math Society}, May 2003, 1-30.

\medskip

[D3] J. D'Angelo,  Hermitian Analysis: from Fourier series to CR Geometry, second edition, Springer, 2019.

\medskip

[DKR] J. D'Angelo, S. Kos, and  E. Riehl,
A sharp bound for the degree of proper monomial mappings between balls, {\it J. Geom. Anal.} 13 (2003), no. 4, 581-593.

\medskip

[DL] J. D'Angelo and J. Lebl, Homotopy equivalence for proper holomorphic mappings,
{\it Adv. Math.} 286 (2016), 160-180.

\medskip

[DL1] J. D'Angelo and J. Lebl, On the complexity of proper mappings between balls, 
{\it Complex Variables and Elliptic Equations},
Volume 54, Issue 3, Holomorphic Mappings (2009), 187-204.

\medskip
[DL2] J. D'Angelo and J. Lebl, Complexity results for CR mappings between spheres,
{\it Int. J. of Math.}, Vol. 20, No. 2 (2009),  149-166.

\medskip

[DX1] J. D'Angelo and M. Xiao, Symmetries in CR complexity theory, {\it Adv. Math.} 313 (2017), 590-627.

\medskip

[DX2] J. D'Angelo and M. Xiao, Symmetries and regularity for holomorphic maps between balls, 
{\it Math. Res. Lett.}, Vol. 25, no. 5 (2018), 1389-1404.

\medskip

[Do1] D. Donoho, Compressed sensing, {\it IEEE Trans. Inform. Theory}, Vol. 52, No. 4 (2006), 1289-1306.
   
\medskip

[Do2]  D. Donoho, For most large underdetermined systems of linear equations the minimal $L^1$-norm solution is also the sparsest solution,
{\it Comm. Pure Appl. Math.},  Vol. 59, no. 6 (2006), 797-829.

\medskip

[F] F. Forstneri\v c,
Extending proper holomorphic maps of positive codimension,
{\it Inventiones Math.}, 95(1989), 31-62.

\medskip

[Fa] J. Faran, Maps from the two-ball to the three-ball, {\it Inventiones Math.} 68 (1982), no. 3, 441-475. 

\medskip

[Gr] D. Grundmeier, Signature pairs for group-invariant Hermitian polynomials, {\it Internat. J. Math.} 22 (2011), no. 3, 311-343. 

\medskip

[GLV] D. Grundmeier, J. Lebl, and L. Vivas, 
Bounding the rank of Hermitian forms and rigidity for CR mappings of hyperquadrics. 
{\it Math. Ann.} 358 (2014), no. 3-4, 1059-1089. 

\medskip

[HJX] X. Huang, S. Ji, and D. Xu, A new gap phenomenon for proper holomorphic mappings from ${\mathbb B}^n$ into ${\mathbb B}^N$, 
{\it Math. Res. Lett.}  13(4) (2006), 515-529.

\medskip

[HJY] X. Huang, S. Ji, and W. Yin,  On the third gap for proper holomorphic maps between balls, {\it Math. Ann.} 358 (2014), no. 1-2, 115-142.

\medskip

[L1] J. Lebl, Normal forms, Hermitian operators, and CR maps of spheres and hyperquadrics,
{\it  Michigan Math. J.} 60 (2011), no. 3, 603-628.

\medskip

[L2] J. Lebl, Addendum to Uniqueness of Certain Polynomials
Constant on A Line, arXiv:1302.1441v2

\medskip

[LL] J. Lebl and D. Lichtblau,
Uniqueness of certain polynomials constant on a line, {\it Linear Algebra Appl.} 433 (2010), no. 4, 824-837.

 \medskip

[LP1]  J. Lebl and H. Peters, Polynomials constant on a hyperplane and CR maps of hyperquadrics,
{\it Mosc. Math. J.} 11 (2011), no. 2, 285-315.

\medskip

[LP2] J. Lebl,  and H. Peters, Polynomials constant on a hyperplane and CR maps of spheres,
 {\it Illinois J. Math.} 56 (2012), no. 1, 155-175.
 
 \medskip
 
 [PWZ] M. Petkovsek, H. Wilf, and D. Zeilberger, $A=B$. A. K. Peters, Ltd., Wellesley, MA, 1996.
 
 \medskip
 
[W] S. Wono, Proper Holomorphic Mappings in Several Complex Variables, PhD thesis, Univ. Illinois, 1993.

\end{document}